\documentclass[12pt]{amsart}

\usepackage{amstext,amssymb,amsmath,stmaryrd,amsthm}
\usepackage{geometry} 
\usepackage{color}
\usepackage{hyperref}
\makeindex

\numberwithin{equation}{section}

\newtheorem{theorem}{Theorem}[section]
\newtheorem{cor}[theorem]{Corollary}
\newtheorem{Lemma}[theorem]{Lemma}

\usepackage{amssymb}
\usepackage{amsmath}
\usepackage{color}
\usepackage{graphicx}
\usepackage{comment}
\usepackage{float}
\usepackage{hyperref}

\def\N{{\mathbb N}}

\def\I{\textbf{I}}
\def\II{\textbf{II}}
\def\III{\textbf{III}}

\newcommand{\R}{\mathbb{R}}

\newcommand{\E}{\mathbb{E}}

\newcommand{\Ge}{\mathcal{L}}
\newcommand{\X}{\mathcal{I}} 

\begin{document}

\title[Poincar\'e type  inequalities for compact PJMP.]{Poincar\'e type  inequalities for compact degenerate   pure jump Markov processes.}

\author{Pierre Hodara  \and  Ioannis Papageorgiou}{ Pierre Hodara$^{*}$  \and  Ioannis Papageorgiou $^{**}$  \\  IME, Universidade de Sao Paulo }

 \thanks{\textit{Address:} Neuromat, Instituto de Matematica e Estatistica,
 Universidade de Sao Paulo, 
 rua do Matao 1010,
 Cidade Universitaria, 
 Sao Paulo - SP-  Brasil - CEP 05508-090.
\\ \text{\  \   \      } 
\textit{Email:} $^*$hodarapierre@gmail.com $^{**}$ ipapageo@ime.usp.br, papyannis@yahoo.com  \\
 This article was produced as part of the activities of FAPESP  Research, Innovation and Dissemination Center for Neuromathematics (grant 2013/ 07699-0 , S.Paulo Research Foundation); This article  is supported by FAPESP grant (2016/17655-8) $^{*}$ and  (2017/15587-8)$^{**}$ }



\begin{abstract}
We  aim in proving  Poincar\'e      inequalities for a class of pure jump Markov processes inspired by the model introduced in \cite{G-L} by Galves and L\"ocherbach to describe the behaviour of  interacting  brain neurons. In particular, we consider  neurons with degenerate jumps, i.e. that lose their memory when they spike, while the probability of a spike depends on the actual position and thus the past of the whole neural system.
\end{abstract}
 
\date{}
\maketitle 
 
\section{Introduction}
 
The aim of this paper is to prove Poincar\'e inequalities for the semigroup $P_t$, as well as for the invariant measure,     of the model introduced in \cite{G-L} by  Galves and L\"ocherbach, to    describe the activity of a biological neural network.    What is in particular interesting about the jump process in question is that it is characterized  by degenerate jumps, in the sense that after a particle (neuron) spikes,  it loses its memory by  jumping to zero.  Furthermore,   the probability of a spike of a particular neuron at any time   depends on  its actual position and thus the past of the whole neural system.

For $P_t$ the associated semigroup   at first  we prove some Poincar\'e type inequalities   of the  form
\begin{align*}Var_{ P_{t}}( f (x)) \leq   \alpha(t) \int_0^tP_s \Gamma(f,f)( x)ds +\beta(t) \sum_{i=1}^n\int_0^tP_s \Gamma(f,f)(\Delta_i ( x))ds.
\end{align*}
for any possible starting configuration  $x$. We give here the general form of the type of inequalities investigated in this paper, however in order to avoid to overload this introduction with technical details we postpone the definitions of classical quantities such as the "carr\'e du champ" and other notations used here to subsection  \ref{subsec:PoincaraInequalities}.

 Then, we restrict ourselves to the special case where the initial configuration $x$ in the domain of the invariant measure, and  we derive the stronger Poincar\'e  inequality 
\begin{align*}Var_{ P_{t}}( f (x)) \leq  \alpha(t)P_t\Gamma(f,f)(x)  +\beta \int _0^tP_s\Gamma(f,f)(x)ds.
\end{align*}
Then we show a Poincar\'e inequality for the invariant measure $\pi$
\[Var_{\pi}(f)\leq c\pi \left( \Gamma (f,f)\right).\] 
 Before we describe the model we present the neuroscience framework of the problem.

\subsection{the neuroscience framework}
The  activity of one neuron is described by the evolution of its membrane potential. This evolution presents from time to time a brief and high-amplitude depolarisation called action potential or spike. The spiking probability or rate of a given neuron depends on the value of its membrane potential. These spikes are the only perturbations of the membrane potential that can be transmitted from one neuron to another through chemical synapses. When a neuron spikes, its membrane potential is reset to $0$ and the post-synaptic neurons connected to it receive an additional amount of membrane potential.

From a probabilistic point of view, this activity can be described by a simple point process since the whole dynamic is characterised by the jump times. In the literature, Hawkes processes are often used in order to describe systems of interacting neurons, see \cite{C17}, \cite{D-L-O},  \cite{D-O}, \cite{G-L}, \cite{H-R-R} and \cite{H-L} for example. The reset to $0$ of the spiking neuron provides a variable length memory for the dynamic and therefore point-processes describing these systems are non-Markovian.

On the other hand, it is possible to describe the activity of the network with a process modelling not only the jump times but the whole evolution of the membrane potential of each neuron. This evolution needs then to be specified between the jumps. In \cite{H-K-L} the process describing this evolution follows a deterministic drift between the jumps, more precisely the membrane potential of each neuron is attracted with exponential speed towards an equilibrium potential. This process is then Markovian and belongs to the family of Piecewise Deterministic Markov Processes introduced by Davis (\cite{Davis84} and \cite{Davis93}). Such processes are widely used in probability modeling  of \textit{e.g.}\ biological or chemical phenomena (see \textit{e.g.} \cite{C-D-M-R} or \cite{PTW-10}, see \cite{ABGKZ} for an overview). The point of view we adopt here is close to this framework, but we work without drift between the jumps. We therefore consider a pure jump Markov process and will make use of the abbreviation PJMP in the rest of the present work.

We consider a process $X_t=(X_t^1,...,X_t^N),$ where $N$ is the number of neurons in the network and where for each neuron $i, \, 1 \le i \le N$ and each time $t \in \R_+,$  each variable $X_t^i$ represents the membrane potential of neuron $ i$ at time $t.$ Each membrane potential $X_t^i$ takes value in $\R_+.$ A neuron with membrane potential $x$ ``spikes" with intensity $\phi(x),$ where $\phi:\R_+ \to \R_+ $ is a given intensity function. When a neuron $i$ fires, its membrane potential is reset to $0,$ interpreted as resting potential, while the membrane potential of any post-synaptic neuron $j$ is increased by $W_{i \to j} \geq 0.$ Between two jumps of the system, the membrane potential of each neuron is constant.

Working with Hawkes processes allows to consider systems with infinitely many neurons, as in \cite{G-L} or \cite{H-L}. For our purpose, we need to work in a Markovian framework and therefore our process represents the membrane potentials of the neurons, considering a finite number $N$ of neurons. 

\subsection{the model}

Let $N > 1 $ be fixed and $(N^i(ds, dz))_{i=1,\dots,N}$ be a family of \textit{i.i.d.}\ Poisson random measures on $\R_+ \times \R_+ $ having intensity measure $ds dz.$ We study the Markov process $X_t = (X^{ 1 }_t, \ldots , X^{ N}_t )$
taking values in $\R_+^N$  and solving, for $i=1,\dots,N$, for $t\geq 0$,
\begin{eqnarray}\label{eq:dyn}
X^{ i}_t &= & X^{i}_0  -  \int_0^t \int_0^\infty 
X^{ i}_{s-}  1_{ \{ z \le  \phi ( X^{ i}_{s-}) \}} N^i (ds, dz) \\
&&+    \sum_{ j \neq i } W_{j \to i} \int_0^t\int_0^\infty  1_{ \{ z \le \phi ( X^{j}_{s-}) \}} N^j (ds, dz).
\nonumber
\end{eqnarray}  

In the above equation for each $j\neq i,\, W_{j \to i} \in \R_+$ is the synaptic weight describing the influence of neuron $j$ on neuron $i.$ Finally, the function $\phi:\R_+\mapsto \R_+$  is the intensity function.

The generator of the process $X$ is given  for any test function $ f : \R_+^N \to \R $  and $x \in \R_+^N$  is described by
\begin{equation}\label{eq:generator1}
\Ge f (x ) = \sum_{ i = 1 }^N \phi (x^i) \left[ f ( \Delta_i ( x)  ) - f(x) \right]
\end{equation}
where
\begin{equation}\label{eq:delta1}
(\Delta_i (x))_j =    \left\{
\begin{array}{lcl}
x^{j} +W_{i \to j}  & j \neq i    \text{ and }  x^{i} +W_{i \to j} \leq m   \\  
x^{j}   & j \neq i    \text{ and }  x^{i} +W_{i \to j} > m    \\
0 & j = i 
\end{array}
\right\}
\end{equation}
for some $m>0$. With this definition the process remains inside the compact set 
\begin{equation}\label{eq:defcompactD}
D:=\{x \in \R_+^N: x_i \leq m , \, 1 \leq i \leq N \}. 
\end{equation}
Furthermore, we also assume the following conditions about the intensity function:
\begin{align}\label{ass:dphipos}\phi(x)>cx  \ \text{for} \  x\in\R_+\end{align}
and
\begin{equation}\label{phi2}
\phi(x)\geq \delta
\end{equation} 
for some strictly positive constants $c$ and $\delta.$

The probability for a neuron to spike grows with its membrane potential so it is natural to think of the function $\phi$ as an increasing function. Condition \eqref{ass:dphipos} implies that this growth is at least linear. Condition \eqref{phi2} models the spontaneous activity of the system: whatever the configuration $x$ is, the system will always have a positive spiking rate.

\subsection{Poincar\'e type inequalities}\label{subsec:PoincaraInequalities}
  Our purpose is to show   Poincar\'e type inequalities for our PJMP, whose dynamic is similar to the model introduced in \cite{G-L}. We will investigate Poincar\'e type inequalities   at first  for the semigroup $P_t$ and then for the invariant measure $\pi$. Concerning the semigroup inequality we will study two different cases. The first, the general one, where the system starts from any possible initial configuration. Then, we restrict to initial configurations that belong to the domain of the invariant measure.

   Let us first describe the general framework and define the Poincar\'e inequalities on a discreet setting (see also \cite{SC}, \cite{D-SC}, \cite{W-Y}, \cite{A-L}  and \cite{Chaf}). At first we should note a convention we will widely use. For a  function $f$ and measure $\nu$ we will  write $\nu (f)$ for the expectation of the function $f$ with respect to the measure $\nu$, that is
   $$\nu (f)=\int fd\nu.$$ 
  
   We consider  a Markov process $(X_t)_{t\geq 0}$ which is described by the infinitesimal generator $\Ge$ and the associated  Markov semigroup $P_t f(x)=\E^x(f(X_t))$. For a semigroup and its associated infinitesimal generator we will need the following well know relationships: $\frac{d}{ds}P_s= \Ge P_s =P_s \Ge$ (see for example \cite{G-Z}).
   
   We define $\pi$ to be the invariant measure for the semigroup $(P_t)_{t\geq 0}$ if and only if 
   \[\pi P_t =\pi.\] 
   
   Furthermore, we define the "carr\'e du champ" operator by:
   $$\Gamma (f,g):=\frac{1}{2}(\Ge (fg)-f\Ge g- g\Ge f).$$

For the  PJMP  process  defined above with the specific generator $\Ge$ given by (\ref{eq:generator1}) 
  a simple calculation shows that the carr\'e  du champ takes the following form.
 \begin{align*}\Gamma (f,f)=\frac{1}{2}(\Ge f^2- 2 f\Ge f)= \frac{1}{2}(\sum_{ i = 1 }^N \phi (x^i) \left[ f ( \Delta_i ( x)  )- f(x) \right]^2).
 \end{align*}
 We say that a measure $\nu$   satisfies a Poincar\'e inequality if there exists constant $C >0$ independent of $f$, such that
 \[(SG) \ \ \  \ \ \ Var_{\nu}(f)\leq C\nu (\Gamma(f,f))\]
  where the variance of  a function $f$ with respect  to a measure $\nu$ is defined  with the usual way as:  $Var_{\nu}(f)=\nu(f-\nu(f))^2$. It should be noted that in the case where  the measure $\nu$ is the semigroup $P_t$, then the constant $C$ may  depend on $t$, $C=C(t)$. 
  
 In \cite{W-Y}, \cite{A-L}  and \cite{Chaf},   the Poincar\'e inequality (SG) for $P_t$ has been shown  for some  point processes, for a constant that depends on time $t$, while the stronger log-Sobolev inequality, has been disproved. The general method used in these papers, that will be followed also in the current work, is based on the so called semigroup method which shows the inequality for the semigroup $P_t$.  

The main difficulty here is that, for the pure jump Markov process that we examine in the current paper, the translation property 
$$\E^{x+y}f(z)=\E^{x}f(z+y)$$
used  in \cite{W-Y} and \cite{A-L} does not hold here. This appears to be important because the translation property  is a key element in these papers, since it allows  to bound the carr\'e  du champ by  comparing the mean $\E^{x}f(z)$ where the process  starts from  position $x$ with the mean $\E^{\Delta_i(x)}f(z)$ where it starts from $\Delta_i(x),$ the jump-neighbour of $x.$ However, we can still obtain Poincar\'e type inequalities, but with a constant $C(t)$ which is a polynomial of order higher than one. This order is higher than the constant  $C(t)=t$, the optimal   obtained in  \cite{W-Y} for a path space of  Poisson point processes. 

It should be noted, that the  the aforementioned translation property relates with the $\Gamma_2$ criterion (see \cite{Ba1} and \cite{Ba2}) for the Poincar\'e inequality (see discussion in subsection \ref{subsec2.2}). Since this is not satisfied in our case we obtain a Poincar\'e type inequality instead.

Before we present the results of the paper it is important to highlight an important distinction on the nature of the initial configuration from which  the process can start. We can classify the initial configurations according to the return probability to them. Recall that the  membrane potential $x_i$ of every neuron $i$ takes positive values within a compact set and that whenever a neuron $j$ different than $i$ spikes, the neuron $i$ jumps $W_{i \to j}$ positions up, while the only other movement it does is jumping to zero when it spikes.  That means that every variable $x_i$ can jump down only to zero while after the first jump, can only pass from a finite number of possible positions. Since the neurons stay still between spikes, that implies that   there is a finite number of possible configurations to which $X=(X^1,...X^N)$ can return after every neuron has spiked for the first time. This is the domain of the invariant measure $\pi$ of the semigroup $P_t$, and we will denote it as $\hat D$. Thus, if the initial configuration $x=(x^1,...x^N)$ does not belong to $\hat D$, after the process enters $\hat D$, it will never return back to this initial configuration. 

It should be noted that it is easy to find initial configurations $x=(x^1,...x^N)\notin \hat D$. For example one can  consider any $x$ such that at least one of the $x^i$s is not a sum of synaptic weights $W_{j \to i}$, or any $x$ with $x^i=x^j$ for every $i$ and $ j$.

   Below we present the     Poincar\'e   inequality   for the semigroup $P_t$ for general starting configurations.
 \begin{theorem} \label{theorem2}Assume the PJMP as described in (\ref{eq:generator1})-(\ref{phi2}). Then, for every $x\in D$, the following Poincar\'e type inequality holds.
\begin{align*}Var_{ E^x}( f (x_t)) \leq   \alpha(t) \int_0^tP_s \Gamma(f,f)( x)ds +\beta \sum_{i=1}^n\int_0^tP_s \Gamma(f,f)(\Delta_i ( x))ds
\end{align*}
with $\alpha(t)$  a second  order polynomial of the time $t$  that does not depend on the function $f$   and $\beta$ a constant. 
\end{theorem}

 One notices that  since the coefficient  $\alpha(t)$  is a polynomial of first order and  $\beta$ is a   constant,  the first term dominates over the second for big time $t$, as shown in the next corollary. 
\begin{cor} \label{corol2}Assume the PJMP as described in (\ref{eq:generator1})-(\ref{phi2}). Then, for every $x\in D$,  and $t$ sufficiently large, i.e. $t>\zeta(f)$ 
\begin{align*} Var_{ E^x}( f (x_t)) \leq  2\alpha(t )\int_0^{t}  P_{w} \Gamma((f,f))(x) dw   \end{align*}
where $\zeta(f)$ is a constant depending only on $f,$
\[
\zeta(f)=\left(\frac{6t_0\sum_{i=1}^nP_t \Gamma(f,f)(\Delta_i ( x))  }{ (1+C_{1}) \int_0^tP_s \Gamma(f,f)( x)ds} \right)^\frac{1}{2}
\] for some positive constants $t_0, \, C_1$ and $M$.
\end{cor}
 One should notice that although the lower value $\zeta(f)$      depends on the   function $f$,   the coefficient $2\alpha (t)$  of the inequality     does not depend on the   function $f$.

  The proof of the  Poincar\'e inequality for the general initial configuration is  presented in  section \ref{local}. 
  
  In the special case where $x$ is on the domain of the invariant measure we obtain the stronger inequality
   \begin{theorem} \label{theorem2b}Assume the PJMP as described in (\ref{eq:generator1})-(\ref{phi2}). Then, there exists a $t_1>0$ such that, for every $t>t_1$ and for every $x\in \hat D$, the following Poincar\'e type inequality holds.
\begin{align*}Var_{ E^x}( f (x_t)) \leq &   \gamma(t)P_t\Gamma(f,f)(x)  +2 \int _0^tP_s\Gamma(f,f)(x)ds.
\end{align*}
with $\gamma(t)$   a third   order polynomial of the time $t$  that do not depend on the function $f$. 
\end{theorem}
As  in the general case, for $t$ large enough we have the following corollary.
\begin{cor}Assume the PJMP as described in (\ref{eq:generator1})-(\ref{phi2}). Then,  there exists a $t_1>0$ such that,   for every for every $x\in \hat D$,  and $t$ sufficiently large, i.e. $t>\max \{\xi(f), t_1 \}$ 
\begin{align*} Var_{ E^x}( f (x_t)) \leq  2 \gamma(t)P_t\Gamma(f,f)(x)   \end{align*}
where $\xi(f)$ is a constant depending only on $f,$
\[
\xi(f)=  \left(\frac{\int _0^tP_s\Gamma(f,f)(x)ds  }{4\theta^2M^2N P_t\Gamma(f,f)(x) } \right)^\frac{1}{3}
\] for some positive constants $\theta$ and $M$.
\end{cor}
    We conclude this section with the Poincar\'e ineqality for the invariant measure $\pi$ presented on the next theorem.
\begin{theorem} \label{InvPoin} Assume the PJMP as described in (\ref{eq:generator1})-(\ref{phi2}). Then  $\pi$ satisfies a Poincar\'e inequality 
  \[\pi\left(f- \pi f \right)^2\leq C_0 \pi(\Gamma(f,f))\]
  for some constant $C_0>0$.
  \end{theorem}
 \section{proof of the  Poincar\'e for general initial configurations. \label{local}}

In this section we focus on   neurons that start with  values on any possible initial configuration $x\in D$  as described by (\ref{eq:generator1})-(\ref{phi2}),  and we prove   the  local Poincar\'e inequalities presented  in Theorem \ref{theorem2} and Corollary \ref{corol2}. Let us first state some technical results.
 
\subsection{Technical results}

We start by  showing properties of  the jump probabilities of the degenerate PJMP processes. Since the process is constant between jumps, the set of reachable positions $y$ after a given time $t$ for a trajectory starting from $x$ is discrete. We  therefore define  
$$\pi_t(x,y):=P_x(X_t=y) \mbox{ and } D_x:=\{y \in D, \pi_t(x,y) > 0 \}.$$ 

This set is finite for the following reasons. On one hand, for each neuron $i \in I,$ the set $S_i= \{0\} \cup \{\sum_{k=1}^n W_{j_k \to i}, n \in \N^*, j_k \in I \}$ is discrete and such that the intersection with any compact is finite. On the other hand, we have $D_x \subset \left[ \prod_{i \in I}  \left(S_i \cup (x_i+S_i) \right) \right] \cap D .$

The idea is that since the process is constant between jumps, elements of $D_x$ are such that there exists a sequence of jumps leading from $x$ to $y.$ Since we are only interested on the arrival position $y,$ among all jump sequences leading to $y,$ we can consider only sequences with minimal number of jumps and the number of such jump sequences leading to positions inside a compact is finite, due to the fact that each $W_{j \to i}$ is non-negative.

Since $x$ is also in the compact $D,$ we can have an upper bound for the cardinal of $D_x$ independent from $x.$ 

For a given time $s \in \R_+$ and a given position $x \in D,$ we denote by $p_s(x)$ the probability that starting at time $0$ from position $x,$ the process has no jump in the interval $[0,s],$ and for a given neuron $i \in I$ by $p_s^i(x)$ the probability that the process has exactly one jump of neuron $i$ and no jumps for other neurons. 

Introducing the notation $\overline{\phi}(x)=\sum_{j \in I} \phi(x_j)$ and given the dynamics of the model, we have that 
$$
p_s(x)= e^{-s \overline{\phi}(x)}
$$
and
\begin{multline}\label{computePsix}
p_s^i(x)= \int_0^s \phi(x_i) e^{-u\overline{\phi}(x)}e^{-(s-u)\overline{\phi}(\Delta^i(x))}du 
\\ = \left\{
\begin{array}{ll}
\frac{\phi(x_i)}{\overline{\phi}(x) - \overline{\phi}(\Delta^i(x))} \left( e^{-s\overline{\phi}(\Delta^i(x))} - e^{-s\overline{\phi}(x) } \right) & \, \mbox{if } \,  \overline{\phi}(\Delta^i(x)) \neq \overline{\phi}(x) \\
s \phi(x_i) e^{-s\overline{\phi}(x) } & \, \mbox{if } \, \overline{\phi}(\Delta^i(x)) = \overline{\phi}(x) 
\end{array}
\right\} .
\end{multline}

Define 
\begin{equation}\label{def:sm}
t_0=  \left\{
\begin{array}{ll}
\frac{ln\left( \overline{\phi}(x) \right) -ln \left( \overline{\phi}(\Delta^i(x)) \right)}{\overline{\phi}(x) - \overline{\phi}(\Delta^i(x))} & \, \mbox{if } \,  \overline{\phi}(\Delta^i(x)) \neq \overline{\phi}(x) \\
\frac{1}{\overline{\phi}(x)} & \, \mbox{if } \, \overline{\phi}(\Delta^i(x)) = \overline{\phi}(x) 
\end{array}
\right\}.
\end{equation}

As a function of $s, \, p_s^i(x)$ is continuous, strictly increasing on $(0, t_0)$ and strictly decreasing on $(t_0,+\infty)$ and we have $p_0^i(x)=0.$

 \begin{Lemma}\label{prop:ctrlsumratio} 
 Assume the PJMP as described in (\ref{eq:generator1})-(\ref{phi2}).There exists positive constants $C_1$ and $C_2$ independent of $t, \, x$ and $y$ such that 
 \begin{itemize}
 \item For all $t> t_0,$ we have 
 $$\sum_{y \in D_x} \frac{\pi_t^2(\Delta^i(x),y)}{\pi_t(x,y)} \leq C_1.$$
 
 \item For all $t \leq  t_0,$ we have
  $$\sum_{y \in D_x\setminus \{\Delta^i(x)\} } \frac{\pi_t^2(\Delta^i(x),y)}{\pi_t(x,y)} \leq C_1$$
and
$$
\frac{\pi_t^2(\Delta^i(x),\Delta^i(x))}{\pi_t(x,\Delta^i(x))} \leq \frac{\pi_t(\Delta^i(x),\Delta^i(x))}{\pi_t(x,\Delta^i(x))} \leq \frac{C_2}{t}.
$$
 
 \end{itemize}

 \end{Lemma}
 
 \begin{proof}
  As said before, the set $D_x$ is finite so it is sufficient to obtain an upper bound for the ratio $\frac{\pi_t^2(\Delta^i(x),y)}{\pi_t(x,y)}.$

We have for all $s \in (0,t)$

\begin{equation}
\frac{\pi_t^2(\Delta^i(x),y)}{\pi_t(x,y)} \leq \frac{\left(\pi_{t-s}(\Delta^i(x),y)p_s(y) + \sup_{z \in D} (1-p_s(z)) \right)^2}{p_s^i(x)\pi_{t-s}(\Delta^i(x),y)}.
\end{equation}

Here we decomposed the numerator according to two events. Either $X_{t-s}=y$ and there no jump in the interval of time $[t-s,t]$ or there is at least one jump in the interval of time $[t-s,t],$ whatever the position $z \in D$ of the process at time $t-s.$ 

From the previous inequality, we then obtain

\begin{equation}\label{forcedJump}
\frac{\pi_t^2(\Delta^i(x),y)}{\pi_t(x,y)} \leq \frac{\left(\pi_{t-s}(\Delta^i(x),y)p_s(y) +  (1- e^{-sN\phi(m)}) \right)^2}{p_s^i(x)\pi_{t-s}(\Delta^i(x),y)}.
\end{equation}

where we recall that the constant $m$ appears in the definition of the compact set $D$ introduced in \eqref{eq:defcompactD}.

Let us first assume that $t > t_0.$ Recall that $ t_0$ is defined in \eqref{def:sm}.

If $\pi_{t- t_0}(\Delta^i(x),y) \geq p_{ t_0}^i(x),$ we have 

\begin{equation}\label{cas1}
\frac{\pi_t^2(\Delta^i(x),y)}{\pi_t(x,y)} \leq  \frac{1}{\left( p_{t_0}^i(x) \right)^2 }.
\end{equation}

Assume now  that $\pi_{t- t_0}(\Delta^i(x),y) < p_{ t_0}^i(x)$  and let us recall that as a function of $s, \, p_s^i(x)$ is continuous, strictly increasing on $(0,  t_0)$ and $p_0^i(x)=0.$

On the other hand, as a function of $s, \, \pi_{t-s}(\Delta^i(x),y)$ is continuous and takes value $\pi_{t}(\Delta^i(x),y) > 0$ for $s=0.$

We deduce from this that there  exists $s_* \in (0, t_0)$ such that $p_{s_*}^i(x) = \pi_{t-s_*}(\Delta^i(x),y).$

Now \eqref{forcedJump} with $s= s_*$ gives us 

\begin{equation}\label{forcedJump2}
\frac{\pi_t^2(\Delta^i(x),y)}{\pi_t(x,y)} \leq (p_{s_*}(y))^2+2 p_{s_*}(y) \frac{1- e^{-s_*N\phi(m)}}{p_{s_*}^i(x)} + \left( \frac{1- e^{-s_*N\phi(m)}}{p_{s_*}^i(x)} \right)^2.
\end{equation}

For all $s \in (0, t_0), \, p_s(y) \leq 1,$ and we then study $\frac{1- e^{-sN\phi(m)}}{p_s^i(x)}$ as a function of $s \in (0,  t_0).$

Using the explicit value of $p_s^i(x)$ given in \eqref{computePsix} and assumption \eqref{ass:dphipos}, we obtain for all $s \in (0, t_0),$

$$
\frac{1- e^{-sN\phi(m)}}{p_s^i(x)} \leq \left\{
\begin{array}{ll}
\frac{e^{ t_0N\phi(m)}}{\delta} \, \frac{\left( \overline{\phi}(x) - \overline{\phi}(\Delta^i(x))\right) \left( 1-e^{-sN\phi(m)}\right)}{1-e^{-s\left( \overline{\phi}(x) - \overline{\phi}(\Delta^i(x))\right)}} & \, \mbox{if } \,  \overline{\phi}(\Delta^i(x)) \neq \overline{\phi}(x) \\ \\
\frac{e^{ t_0N\phi(m)}}{\delta} \frac{1-e^{-s}}{s} & \, \mbox{if } \, \overline{\phi}(\Delta^i(x)) = \overline{\phi}(x) 
\end{array}
\right\}. 
$$

Recall that $\delta > 0$ is defined in assumption (\ref{phi2}) and satisfies $\phi(x) \geq \delta$ for all $x \in \R_+.$ 

In both cases, when $s$ is far from zero, we can obtain an upper bound independent of $x,$ and when $s$ goes to zero, the limit of the right hand term is $\frac{N\phi(m)e^{ t_0N\phi(m)}}{\delta}.$

%

From this, we deduce that there exists a constant $M_D$ such that for all $s \in (0, t_0),$
$$\frac{1- e^{-sN\phi(m)}}{p_s^i(x)} \leq M_D.$$

Putting all together, we obtain the announced result for the case  where $t >  t_0.$

We now consider the case where $t \leq  t_0$. 

We start by considering the case where $y \neq \Delta^i(x)$ and go back to \eqref{forcedJump}.

As a function of $s, \, \pi_{t-s}(\Delta^i(x),y)$ is continuous and takes values $\pi_{t}(\Delta^i(x),y) > 0$ and $\pi_0(\Delta^i(x),y) = 0$ respectively for $s=0$ and $s=t.$ 

We deduce from this that there exists $s_* \in (0,t) \subset (0,t_0)$ such that $p_{s_*}^i(x) = \pi_{t-s_*}(\Delta^i(x),y)$ and we are back in the previous case so that the same result holds.

Let us assume now that $y=\Delta^i(x),$ we have

\begin{equation}\label{Deltaix1}
\frac{\pi_t^2(\Delta^i(x),y)}{\pi_t(x,y)} \leq \frac{\pi_t(\Delta^i(x),\Delta^i(x))}{\pi_t(x,\Delta^i(x))} \leq \frac{1}{p_t^i(x)}.
\end{equation}

Recall the explicit expression of $p_t^i(x)$ given in \eqref{computePsix} and use (\ref{phi2}) to bound the intensity function 

$$
p_t^i(x) =t \phi(x_i) e^{-t\overline{\phi}(x) }\geq t\delta e^{-t_0 \sup_{x\in D}\phi(x)}=Ct
$$

for some constant $C$ independent of $t$ and $x,$ which gives us the announced result.

 \end{proof}

Taking under account  the last   result, we can obtain the first technical bound needed in the proof of  the local Poincar\'e inequality.

\begin{Lemma}\label{techn1}  Assume the PJMP as described in (\ref{eq:generator1})-(\ref{phi2}). Then
 \begin{align*}\left(\int_{t_0}^{t-s}(\E^{\Delta_i ( x) }-\E^x) \sum_{ j = 1 }^N \phi (x_{u}^j)(f(\Delta_j ( x_{u}))-f(x_{u})) du\right)^2 \leq  \\    (t-s)(1+C_{1})M\int_{t_0}^{t-s} \E^x\Gamma((f,f))(x_{u}). 
 \end{align*}
\end{Lemma}
\begin{proof}
 Consider $\pi_t(x,y)$ to probability kernel of $\E^x$, i.e. $\E^x (f(x_t))=\sum_y \pi_t(x,y) f(y) $. Then we can write
 \begin{align*}\nonumber
 \nonumber \II_1:=&\left(\int_{t_0}^{t-s}(\E^{\Delta_i ( x) }-\E^x) \sum_{j = 1 }^N \phi (x_{u}^j)(f(\Delta_j( x_{u}))-f(x_{u})) du\right)^2=   \\   =& \left( \int_{t_0}^{t-s}\sum_y \pi_u(x,y)\left((\frac{\pi_u(\Delta_i (x),y)}{\pi_u(x,y)}-1) \sum_{j = 1 }^N \phi _u(y^j)(f(\Delta_j(y))-f(y))\right)du\right)^2. \nonumber\end{align*}
To continue we will use  Holder's inequality to pass the second power inside the first integral, which will give 
 \begin{align*}\nonumber
  & \II_1\leq \\ &
(t-s) \int_{t_0}^{t-s} \left(\sum_y \pi_u(x,y)\left((\frac{\pi_u(\Delta_i (x),y)}{\pi_u(x,y)}-1) \sum_{ j = 1 }^N \phi_u (y^j)(f(\Delta_j ( y))-f(y))\right)\right)^2 du\end{align*}
and then  we apply the Cauchy-Schwarz inequality for the measure $\E^x$ to get
 \begin{align*}
 \nonumber& \II_1 \leq  \\ &
(t-s)\int_{t_0}^{t-s}\E^x\left(\frac{\pi_u(\Delta_i (x) ,y)}{\pi_u(x,y)}-1\right)^2*\E^x\left( \underbrace{ \sum_{ j = 1 }^N \phi (y_u^j)(f(\Delta_j( y_u))-f(y_u))}_{:=\mathbb{S}}\right)^2  du.\end{align*}
   The first quantity involved in the above integral is bounded  from Lemma \ref{prop:ctrlsumratio}  by a constant  
   \[\E^x\left(\frac{\pi_u(\Delta_i (x) ,y)}{\pi_u(x,y)}-1\right)^2\leq 1+ E^x\left(\frac{\pi_u(\Delta_i (x) ,y)}{\pi_u(x,y)}\right)^2=1+\sum_{y \in D_x} \frac{\pi_u^2(\Delta^i(x),y)}{\pi_u(x,y)}\leq 1+C_1\]while for the sum involved in the second quantity, since $\phi(x)\geq \delta>0$ we can use Holder's inequality 
\begin{align*}\mathbb{S}&=\left(\sum_{ j = 1 }^N \phi (y_u^j)\right)^2\left(  \sum_{ j = 1 }^N \frac{\phi (y_u^j)}{\sum_{ j = 1 }^N \phi (y_u^j)}(f(\Delta_j( y_u))-f(y_u))\right)^2\leq \\ & \leq M \sum_{ j = 1 }^N \phi (y_u^j)(f(\Delta_j(y_u))-f(y_u))^2   =  M\Gamma (f,f)(x_u)\end{align*}where $M=\sup_{x\in D} \sum_{ i = 1 }^N \phi (x^i)$, so that
\begin{align*} \II_1     \leq   (t-s)(1+C_{1})M\int_{t_0}^{t-s} \E^x\Gamma((f,f))(x_{u}) du.\end{align*}
\end{proof}
We will now extend the last bound  to an integral on a time domain starting at $0.$ 
\begin{Lemma}\label{techn} For  the PJMP as described in (\ref{eq:generator1})-(\ref{phi2}), we have
 \begin{align*}\left(\int_0^{t-s}\left(\E^{\Delta_i ( x) }( \Ge f(x_{u}))-  \E^x(\Ge f(x_{u}))\right)du\right)^2 \leq & 16t^2_0M\Gamma(f,f)(\Delta_i ( x))+\\   &
 +c(t-s)\int_0^{t-s} \E^x\Gamma((f,f))(x_{u}) du 
 \end{align*}
 where $c(t)=t_0 8M(C_1+1)+ 2t(1+C_{1})M$. 
\end{Lemma}
\begin{proof} 
In order to calculate a bound for 
\begin{align}\nonumber
&\E^{\Delta_i ( x) }( \Ge f(x_{u}))-  \E^x(\Ge f(x_{u}))\nonumber  
 \end{align}
 we will need to control the ration $\frac{\pi_u^2(\Delta_i(x),y)}{\pi_u(x,y)}$.  As shown in Lemma \ref{prop:ctrlsumratio},  this ratio  depends on time when $u\leq t_0$, otherwise it is bounded by a constant. For this reason we will start by breaking the integration variable of the time $t$ into two domains, $(0,t_0)$ and $(t_0,t-s)$.
   \begin{align}\nonumber
\I_1:=&\left(\int_0^{t-s}\left(\E^{\Delta_i ( x) }( \Ge f(x_{u}))-  \E^x(\Ge f(x_{u}))\right)du\right)^2\leq \\ \nonumber      &2\underbrace{\left(\int_{t_0}^{t-s}(\E^{\Delta_i ( x) }-\E^x) \sum_{ i = 1 }^N \phi (x_{u}^i)(f(\Delta_i ( x_{u})-f(x_{u})) du\right)^2}_{\II_1} +
\\  &+2\underbrace{
\left(\int_0^{t_0}(\E^{\Delta_i ( x) }-\E^x) \sum_{ i = 1 }^N \phi (x_{u}^i)(f(\Delta_i ( x_{u}))-f(x_{u})) du\right)^2  }_{:=\II_2} .\label{pr2.4}
 \end{align}
The first   summand $\II_1$ is upper bounded by the previous lemma. To bound the second term $\II_2$ 
  on the right hand side of (\ref{pr2.4}) we write
 \begin{align}\nonumber
&\II_2    \leq  \\   &  \nonumber  2\underbrace{\left(\int_0^{t_0}( \pi_u(\Delta_i ( x),\Delta_i ( x)) -\pi_u(x,\Delta_i ( x)) ) \sum_{ i = 1 }^N \phi (\Delta_i ( x)^i)(f(\Delta_i ( \Delta_i ( x))-f(\Delta_i ( x))) du\right)^2}_{:=\III_1} + 
\\   &    +2\underbrace{\left(\int_0^{t_0}(\sum_{y\in D, y\not=\Delta_i (x)}(\pi_u(\Delta_i (x),y)-\pi_u(x,y))\sum_{ i = 1 }^N \phi (y^i)(f(\Delta_i ( y)-f(y)) du\right)^2}_{:=\III_2}.  \label{pr2.4+1}
\end{align} 
The   distinction on the two cases, whether after time $u$ the neurons configuration is $\Delta_i(x)$ or not,  relates to the two different bounds Lemma  \ref{prop:ctrlsumratio} provides  for the fraction $ \frac{\pi_t^2(\Delta^i(x),y)}{\pi_t(x,y)} $ whether $y$ is $ \Delta_i(x) $ or not. We will first calculate the second term on the right hand side. For this term  we will work similar to Lemma \ref{techn1}.  At first we will apply  the Holder inequality on the time integral after we first divide with the  normalisation constant $t_0$. This will give   
  \begin{align*}\nonumber
& \nonumber \III_2  \leq\\ &  \nonumber t_0 \int_0^{t_0}\left(\sum_{y\in D, y\not=\Delta_i (x)}(\frac{\pi_u(\Delta_i (x),y)}{\pi_u(x,y)}-1)\pi_u(x,y)
\sum_{ i = 1 }^N \phi (y^i)(f(\Delta_i ( y))-f(y)\right) ^2du. 
\end{align*} 
 Now we will use the Cauchy-Schwarz inequality  in the first sum. We will then obtain the following
  \begin{align}\nonumber
\III_2 \leq &  \nonumber t_0 \int_0^{t_0}\left[\sum_{y\in D, y\not=\Delta_i (x)}\pi_u(x,y)(\frac{\pi_u(\Delta_i (x),y)}{\pi_u(x,y)}-1)^2\right]
* \\   & \left[ \sum_{y\in D, y\not=\Delta_i (x)}\pi_u(x,y) \left( \sum_{ i = 1 }^N \phi (y^i)(f(\Delta_i ( y))-f(y)\right)^2\right]du.  \label{pr2.7} \end{align} 
The first term on the last product   can    be upper bounded from Lemma  \ref{prop:ctrlsumratio} 
  \begin{align}
   \nonumber \sum_{y\in D, y\not=\Delta_i (x)}\pi_u(x,y)(\frac{\pi_u(\Delta_i (x),y)}{\pi_u(x,y)}-1)^2\leq &\sum_{y\in D, y\not=\Delta_i (x)}(\frac{\pi_u^2(\Delta_i (x),y)}{\pi_u(x,y)}+\pi_u(x,y)) \\  \leq & (C_1+1).
 \label{pr2.8} \end{align} 
 While for the second term involved in the product of (\ref{pr2.7}) we can write
  \begin{align*}\nonumber
 \nonumber  \sum_{y\in D, y\not=\Delta_i (x)}&\pi_u(x,y) \left( \sum_{ i = 1 }^N \phi (y^i)(f(\Delta_i ( y))-f(y)\right)^2 \leq \\ & \nonumber \leq \sum_{y\in D, y\not=\Delta_i (x)}\pi_u(x,y)   ( \sum_{ i = 1 }^N \phi (y^i)) \sum_{ i = 1 }^N \phi (y^i)(f(\Delta_i (y))-f(y))^2\leq  \\ &  \leq M\sum_{y}\pi_u(x,y)  \Gamma(f,f)(y) = \\ & =M\E^x \Gamma(f,f)(x_u)
 \nonumber \end{align*}  
 where for the first bound we made use once more of the Holder inequality, after we divided with the appropriate normalisation constant  $\sum_{ i = 1 }^N \phi (x_{u}^i)$.  If we put   the last  bound together  with  (\ref{pr2.8}) into   (\ref{pr2.7}), we obtain
   \begin{align}
\III_2 \leq &  t_0M(C_1+1) \int_0^{t_0}
\E^x \Gamma(f,f)(x_u)   du.  \label{pr2.8+} \end{align} 

 We now calculate the first summand  of (\ref{pr2.4+1}). Notice that in this case we cannot use the analogue bound from Lemma \ref{prop:ctrlsumratio},  that is $
\frac{\pi_t^2(\Delta^i(x),\Delta^i(x))}{\pi_t(x,\Delta^i(x))} \leq   \frac{C_2}{t},
$ as we did for $\III_2$,  since   that will lead to a final upper bound $\III_1\leq   t_0M(C_1+1) \int_0^{t_0}\frac{1}{u}
\E^x \Gamma(f,f)(x_u)   du $ which may diverge. Instead, we will bound the $\III_1$ by the carr\'e du champ of the function after the first jump.  We can write
\begin{align*}
\III_1     \leq     &  \nonumber  4\left(\int_0^{t_0}(\sum_{ i = 1 }^N \phi (\Delta_i ( x)^i)\vert f(\Delta_i ( \Delta_i ( x)))-f(\Delta_i ( x))\vert du\right)^2 \\ \leq  &  \nonumber  4t_0^2( \sum_{ i = 1 }^N \phi (\Delta_i ( x)^i))^2\left(\sum_{ i = 1 }^N \frac{\phi (\Delta_i ( x)^i)}{\sum_{ i = 1 }^N \phi (\Delta_i ( x)^i)}\vert f(\Delta_i ( \Delta_i ( x)))-f(\Delta_i ( x))\vert \right)^2   \end{align*}
where above we divided with the   normalisation constant $\sum_{ i = 1 }^N \phi (\Delta_i ( x)^i)$, since $\phi(x)\geq \delta$.  We can now    apply the Holder inequality on the sum,   so that 
  \begin{align}\nonumber
\III_1     \leq      
  4t^2_0M(\sum_{ i = 1 }^N \phi (\Delta_i ( x)^i) (f(\Delta_i ( \Delta_i ( x)))-f(\Delta_i ( x))^2)  =
  4t^2_0M\Gamma(f,f)(\Delta_i ( x)).
\end{align} 
 If we combine this  together with (\ref{pr2.8+})  and  (\ref{pr2.4+1}) we get the following bound for the second term of (\ref{pr2.4})
 \begin{align}\nonumber
\II_2 \leq   \nonumber 8t^2_0M\Gamma(f,f)(\Delta_i ( x))    +4(C_1+1) t_0M \int_0^{t_0}
\E^x \Gamma(f,f)(x_u)du. 
\end{align}  

The last one together with the bound shown in  Lemma \ref{techn1} for the first term $\II_1$ of  (\ref{pr2.4}) gives
 \begin{align*} \I_1\leq & t^2_016M\Gamma(f,f)(\Delta_i ( x))+t_0 8M(C_1+1)\int_0^{t_0}\E^x \Gamma(f,f)(x_u)du+\\ &  +2 (t-s)(1+C_{1})M\int_{t_0}^{t-s} \E^x\Gamma((f,f))(x_{u}) du \leq \\ \leq & t^2_016M\Gamma(f,f)(\Delta_i ( x))
 +2M(C_1+1)(4t_0  +   (t-s))\int_0^{t-s} \E^x\Gamma((f,f))(x_{u}) du. 
 \end{align*}
  since  the carr\'e  du champ is non negative, as shown below
\begin{align*} \Gamma (f,f)=\frac{1}{2}(\Ge (f^2)- 2 f\Ge f)=\lim_{t\downarrow 0} \frac{1}{2t} P_t f^2 - (P_t f)^2\geq 0
\end{align*}
by Cauchy-Swartz inequality. 
\end{proof}
We have obtained all the technical results that we need in order to show the Poincar\'e inequality for the semigroup $P_t$ for general initial configurations. 

\subsection{proof of Theorem \ref{theorem2}}\label{subsec2.2}

˜

  Denote $P_tf(x)=\E^xf(x_t)$. Then
 \begin{align}
 P_t f^2(x)-(P_t f(x))^2=\int_0^t \frac{d}{ds}P_s (P_{t-s}f)^2(x) ds=\int_0^t P_s \Gamma  (P_{t-s}f,P_{t-s}f)(x) ds \label{pr2.1}
 \end{align}
 since $\frac{d}{ds}P_s= \Ge P_s =P_s \Ge$.

 We can write 
  \begin{align}
\Gamma  (P_{t-s}f,P_{t-s}f)(x) = \sum_{ i = 1 }^N \phi (x^i) ( \E^{\Delta_i ( x) }f(x_{t-s})-\E^{x}f(x_{t-s}))^2. 
\label{pr2.2} \end{align}

 If we could use the translation property $\E^{x+y}f(z)=\E^{x}f(z+y)$ used for instance in proving Poincar\'e and modified log-Sobolev inequalities in \cite{W-Y} and \cite{A-L}, then we could bound relatively easy the carr\'e  du champ of the expectation of the functions by the carr\'e  du champ of the functions themselves, as demonstrated below
\begin{align*}\sum_{ i = 1 }^N \phi (x^i)( \E^{\Delta_i ( x) }f(x_{t-s})-\E^{x}f(x_{t-s}))^2=& \sum_{ i = 1 }^N \phi (x^i) (\E^{x }f(\Delta_i (x_{t-s}))-\E^{x}f(x_{t-s}))^2 \\ \leq & P_{t-s}\Gamma(f,f)(x) \end{align*}
The inequality $\Gamma (P_tf,P_tf)\leq P_t\Gamma (f,f)$ for $t>0$ relates directly with the $\Gamma_2$ criterion  (see \cite{Ba1} and \cite{Ba2}) which states that if $\Gamma_2 (f):=\frac{1}{2}\Ge (\Gamma(f,f))-2\Gamma(f,\Ge l)\geq 0$ then the  Poincar\'e inequality is true, since 
\begin{align*}\frac{d}{ds}(P_s \Gamma(P_{t-s}f,P_{t-s}f))=&\frac{1}{2}P_s (\Ge \Gamma(P_{t-s}f,P_{t-s}f))-2\Gamma (P_{t-s}f,\Ge P_{t-s}f)= \\ =&P_s (\Gamma_2(P_{t-s}f))\geq 0\end{align*}
implies $\Gamma (P_tf,P_tf)\leq P_t\Gamma (f,f)$ (see also \cite{A-L}).

 Unfortunately this is not the case with our PJMP  where the degeneracy of jumps and the memoryless nature of them allows any  neuron $x_i$ to jump to zero from any position, with a probability that depends on the current configuration of the neurons. Moreover, contrary on the case of Poisson processes, our intensity also depends on the position.

 In order to obtain the carr\'e du champ  of the functions we will make use of the Dynkin's formula which will allow us to bound the expectation of a function with the expectation of the infinitesimal generator of the function which  is comparable to the desired  carr\'e  du champ of the function.

So, from Dynkin's formula 
$$\E^x f(x_t)=f(x)+\int_0^{t}\E^x(\Ge f(x_u))du$$we get
  \begin{align*}\nonumber
\left(\E^{\Delta_i ( x)}f(x_{t-s})-\E^{x}f(x_{t-s})\right)^2   \leq &2  \left(f(\Delta_i ( x) )-  f(x)\right)^2 +\\ & +2 \left(\int_0^{{t-s}}\left(\E^{\Delta_i ( x) }( \Ge f(x_{u}))-  \E^x(\Ge f(x_{u}))\right)du\right)^2. 
 \end{align*}
  In order to bound the second term above we will use the bound shown in  Lemma \ref{techn}  
     \begin{align*}\nonumber
\left(\E^{\Delta_i ( x)}f(x_{t-s})-\E^{x}f(x_{t-s})\right)^2   \leq & 2  \left(f(\Delta_i ( x) )-  f(x)\right)^2 +32t^2_0M\Gamma(f,f)(\Delta_i ( x))+\\ 
 &+2c(t-s)\int_0^{t-s} \E^x\Gamma((f,f))(x_{u}) du 
 \end{align*}
 This together with (\ref{pr2.2}) gives 
  \begin{align*}
\Gamma  (P_{t-s}f,P_{t-s}f)(x) \leq &2\Gamma(f,f)( x) +32t^2_0M\sum_{i=1}^n \phi(x^i)\Gamma(f,f)(\Delta_i ( x))+\\ 
 &+2Mc(t-s)\int_0^{t-s} \E^x\Gamma((f,f))(x_{u}) du. 
 \end{align*}
Finally, plugging this in (\ref{pr2.1}) we obtain \begin{align*}
 P_t f^2(x)-(P_t f(x))^2\leq & 2\int_0^tP_s\Gamma(f,f)( x)ds +32t^2_0M\int_0^tP_s\sum_{i=1}^n \phi(x^i)\Gamma(f,f)(\Delta_i ( x))ds+\\ 
 &+2M\int_0^tc(t-s)P_{s}\int_0^{t-s} \E^x\Gamma((f,f))(x_{u}) du ds. 
 \end{align*}
   For the second term we can bound $\phi$ by $M$. For the last term on the right hand side, since the carr\'e du champ is non negative,  we can get
\begin{align*}P_s\int_0^{t-s} \E^x\Gamma((f,f))(x_{u}) du &=P_s\int_s^{t}  P_{w-s}\Gamma((f,f))(x) dw\leq \\ & \leq P_s \int_0^{t }  P_{w-s}\Gamma((f,f))(x) dw=\int_0^{t }  P_{w}\Gamma((f,f))(x) dw.
\end{align*}
where above we used  the   property  of the Markov  semigroup  $P_s P_{w-s}=P_w$ . Since this last quantity does not depend on $s$, and $c(t-s)\leq c(t)$ we further get
\begin{align*}\int_0^t c(t-s)P_s\int_0^{t-s} \E^x\Gamma((f,f))(x_{u}) duds \leq c(t)t\int_0^{t }  P_{w}\Gamma((f,f))(x) dw.
\end{align*}
Putting everything together we finally obtain
\begin{align*}
 P_t f^2(x)-(P_t f(x))^2\leq & (2+2Mc(t)t)\int_0^tP_s\Gamma(f,f)( x)ds \\  &+32t^2_0M^2\sum_{i=1}^n\int_0^tP_s \Gamma(f,f)(\Delta_i ( x))ds. 
 \end{align*}
 And so,  the theorem follows for constants 
\begin{align*}\label{const1} \alpha(t)=2+2Mtc(t)=2+2Mtt_0 8M(C_1+1)+ 4t^2(1+C_{1})M^2  \ \text{and}   \   \beta=32t^2_0M^2.
\end{align*}

\section{proof of the Poincar\'e inequalityies for starting configuration on the domain of the  invariant measure}\label{main3n}

˜

We start by showing that in the case where the initial configuration belongs on the domain $x\in \hat D$ of the invariant measure,  we obtain a strong lower bound for the probabilities $\pi_t(x,y)=P_x(X_t=y) $, for time $t$ big enough, as  presented on the following lemma.

 \begin{Lemma}\label{boundProb} 
 Assume the PJMP as described in (\ref{eq:generator1})-(\ref{phi2}). Then, for every $x\in \hat D$ and $y\in D_x$
 \[\pi_t(x,y)\geq \frac{1}{\theta}\] for $t\geq t_1=\frac{1}{\delta}+t_0$. \end{Lemma}
 
 \begin{proof} 
Since $\hat D\subset D$ is finite, we know that there exists a strictly positive  constant $\eta$, such that $\pi (x)>\eta>0$ for every $x\in \hat D$. Since, $\lim_{t\rightarrow \infty }\pi_t(x,y)=\pi (y)$ for every $x\in \hat D$, such that $y\in D_x$, we obtain that there exists a $\theta>0$ such that $\pi_t(x,y)>\frac{1}{\theta}$ for every $x \in \hat D$, such that $y\in D_x$, since $\hat D \subset D$ is finite.
 \end{proof}

Taking under account  the last result, we can obtain the first technical bound needed in the proof of  the local Poincar\'e inequality, taking  advantage of the bounds shown for times bigger than $t_1$. 

\begin{Lemma}\label{N1techn}  Assume the PJMP as described in (\ref{eq:generator1})-(\ref{phi2}). Then, for every $z\in \hat D$
 \begin{align*}P_s\left(\int_0^{t-s}\left(\E^{\Delta_i ( z) }( \Ge f(z_{u})\X_{z_u\in \hat D})-  \E^z(\Ge f(z_{u})\X_{z_u\in \hat D})\right)du\right)^2 \leq \\  4\theta^2 t^2 M  \E^x ( \Gamma(f,f)(x_t)\X_{x_t\in \hat D}) 
 \end{align*}
  for every  $t\geq t_1$. 
\end{Lemma}
\begin{proof} 
We can compute
   \begin{align}\nonumber
\I_2:=&P_s\left(\int_0^{t-s}\left(\E^{\Delta_i ( z) }( \Ge f(z_{u})\X_{z_u\in \hat D})-  \E^z(\Ge f(z_{u})\X_{z_u\in \hat D})\right)du\right)^2\leq \\ \nonumber  \leq    & 2P_s\left(\int_0^{t }\sum_{y\in \hat D}\pi_u(\Delta_i (z),y) \vert\Ge f(y)\vert du\right)^2+2P_s \left(\int_0^{t}\sum_{y\in \hat D}\pi_u(z,y)\vert \Ge f(y)\vert du\right)^2
\end{align} 
 Now we will use three times the Cauchy-Schwarz inequality, to pass the square inside the integral and the two sums. We will then obtain  
    \begin{align}\nonumber
\I_2\leq  &  \nonumber    2\theta^2 t  M\sum_{\omega=z,\Delta_i(z)} \sum_{y\in \hat D}\int_0^{t }P_s\pi_u(\omega,y)   \sum_{ i = 1 }^N \phi (y^i)\left(f(\Delta_i ( y))-f(y)  \right)^2du\\  \nonumber  =& 2\theta^2 t  M\sum_{\omega=z,\Delta_i(z)} \sum_{y\in \hat D}\int_0^{t } \pi_{s+u}(\omega,y)   \sum_{ i = 1 }^N \phi (y^i)\left(f(\Delta_i ( y))-f(y)  \right)^2du\end{align}


where above we used the semigoup property $P_sP_u=P_{s+u}$. Since $t\geq t_1$, we can use Lemma \ref{boundProb} to bound for every $w$ and $y\in \hat D$,  $\pi_{u+s}(w,y)\leq \theta \pi_t(x,y)$. We then obtain 
   \begin{align}
\I_2&\leq   \nonumber   4\theta^2 t^2M \sum_{y\in D}\pi_t(x,y) \sum_{ i = 1 }^N \phi (y^i)\left(f(\Delta_i ( y))-f(y)  \right)^2\\ & = \nonumber   4\theta^2 t^2MP_t\Gamma(f,f)(x).
\end{align}
\end{proof}

\subsection{proof of Theorem \ref{theorem2b}}\label{subsecN1.3}

˜

We can now show  show the Poincar\'e inequality for the semigroup $P_t$ for   initial configurations inside the domain $\hat D$ of the invariant measure $\pi$. 
\begin{proof}
We will work as in the proof of the Poincar\'e inequality of Theorem \ref{theorem2} for general initial conditions. As before, we denote $P_tf(x)=\E^xf(x_t)$. Then
 \begin{align}\nonumber
  \E^x (f^2(x_t)\X_{x_t \in \hat D}) -\left( \E^x ( f(x_t)\X_{x_t \in \hat D})   \right)^2&=\int_0^t \frac{d}{ds}P_s \left( \E^{x_s} ( f(x_{t-s})\X_{x_{t-s} \in \hat D})   \right)^2(x) ds=\\   =\int_0^t P_s \Gamma  ( \E^{x_s} ( f(x_{t-s})&\X_{x_{t-s} \in \hat D}) , \E^{x_s} ( f(x_{t-s})\X_{x_{t-s} \in \hat D}) )(x) ds \label{nnpr2.1}
 \end{align}
 since $\frac{d}{ds}P_s= \Ge P_s =P_s \Ge$.  To bound the carr\'e du champ, as in the general case of Theorem \ref{theorem2}, from (\ref{pr2.2}) and Dynkin's formula we obtain the following
  \begin{align*}\nonumber
\Gamma  (\E^{x_s} ( f(&x_{t-s}) \X_{x_{t-s} \in \hat D}) ,\E^{x_s} ( f(x_{t-s})\X_{x_{t-s} \in \hat D}) )  \leq  2 \Gamma(f,f)(x_s) +\\   \nonumber &+2 \sum_{ i = 1 }^N \phi (x_s^i) \left(\int_0^{{t-s}}\left(\E^{\Delta_i ( x_s) }( \Ge f(x_{u}) \X_{x_{u} \in \hat D})-  \E^{x_s}(\Ge f(x_{u}) \X_{x_{u} \in \hat D})\right)du\right)^2
. \end{align*}
From that and (\ref{nnpr2.1}) and the bound  $M=\sup_{x\in D} \sum_{ i = 1 }^N \phi (x^i)$ on $\phi$, we then  get
\begin{align*}  \E^x (f^2(x_t)\X_{x_t \in \hat D}) -&\left( \E^x ( f(x_t)\X_{x_t \in \hat D})   \right)^2\leq 2\int _0^tP_s\Gamma(f,f)(x)ds+ \\ +2M\sum_{ i = 1 }^N\int_0^t P_s  &\left(\int_0^{{t-s}}\left(\E^{\Delta_i ( x_s) }( \Ge f(x_{u}) \X_{x_{u} \in \hat D})-  \E^{x_s}(\Ge f(x_{u}) \X_{x_{u} \in \hat D})\right)du\right)^2ds.
 \end{align*} 
Since $t\geq t_1$,  we can use  Lemma   \ref{N1techn} to bound the second term on the right hand side 
 \begin{align*}  \E^x (f^2(x_t)\X_{x_t \in \hat D}) -\left( \E^x ( f(x_t)\X_{x_t \in \hat D})   \right)^2\leq&2 \int _0^tP_s\Gamma(f,f)(x)ds \\ & +8\theta^2M^2Nt^3P_t\Gamma(f,f)(x).
 \end{align*} 
 \end{proof}

\section{proof of the Poincar\'e inequalities for the invariant measure}\label{main3.2}

˜

In this section  we  prove a Poincar\' e inequality for the invariant measure $\pi$ presented in Theorem \ref{InvPoin}, using methods developed in   \cite{B-C-G},  \cite{B-G-L} and \cite{C-G-W-W}.

\begin{proof}  At first assume $\pi (f)=0$. We can write
\begin{align*}Var_{\pi}(f)=\int  f^{2}d\pi=\int   f^{2}\X_{\hat D}d\pi .  \end{align*}
We will   follow the method from \cite{SC} used to prove  Spectral Gap for finite Markov chains. Since $\pi(f)=0$,  we can write
\[\int   f^{2}\X_{\hat D}d\pi=\frac{1}{2}\int \int ( f(x)-f(y))^{2}\X_{x\in \hat D}\X_{y\in \hat D}\pi(dx)\pi(dy).\]
Consider $\delta(x,y)=\{ i_{1},i_{2},..., i_{\vert \gamma(x,y)\vert}\}$ the shortest path from  $x\in \hat D$ to $y\in \hat D$, where the  indexes $i_k$ stand for the neuron that spikes. Since $\hat D$ is finite $ \max \{x,y \in \hat D:\vert\delta (x,y)\vert\}$ is finite. We denote   $\tilde x^0=x$ and $\tilde x^{k}=\Delta_{i_k}(\Delta_{i_{k-1}}(...\Delta_{i_1}(x))...)$, for $k=1, ...,\vert \delta (x,y)\vert$, so that $\tilde x^{\vert \delta(x,y)\vert}=y$. So we can write
 \begin{align*}\pi(x)\pi(y) ( f(x)-f(y))^2\leq &\pi (y)\pi (x)  \sum_{j=0}^{\vert \delta(x,y)\vert}( f(\Delta(\tilde x^j)_{i_j})-f(\tilde x^j))^2\\ 
 \leq & \frac{ \pi (y)\pi (x)}{ \delta}  \sum^{\vert \delta(x,y)\vert}_{j=0}\varphi(\tilde x^i_{i_j})( f(\Delta(\tilde x^j)_{i_j})-f(\tilde x^j))^2\end{align*}
 where above we used  that $\phi\geq \delta$. We can now form the car\'e du champ
 \begin{align*}\pi(x)\pi(y) ( f(x)-f(y))^2\leq &\frac{ \pi (y)\pi (x)}{ \delta} \sum^{\vert \delta(x,y)\vert}_{j=0}\sum_{i\in D}\varphi(\tilde x^i_{i_j})( f(\Delta(\tilde x^j)_i)-f(\tilde x^j))^2
\\ 
\leq & \frac{ \pi (y)\pi (x)}{ \min \{x\in \hat D:\pi(x)\}\delta}  \sum^{\vert \delta(x,y)\vert}_{j=0}\pi ((\tilde x^j))\Gamma (f,f)(\tilde x^j).\end{align*}
 We then  have 
\begin{align*}\int   f^{2}\X_{D}d\pi\leq  &\frac{N^2}{2\min \{x\in \hat  D:\pi(x)\}\delta}  \sum_{x\in D}\pi (x)\Gamma (f,f)(x) \\  =&\frac{N^2}{2\min \{x\in \hat D:\pi(x)\}\delta} \pi(\Gamma (f,f)\X_{\hat D}).\end{align*}
Putting all together leads to 
\begin{align*}Var_{\pi}(f)\leq\frac{N^2}{2\min \{x\in \hat D:\pi(x)\}\delta}\int \Gamma(f,f) d\pi.  \end{align*}
\end{proof}

 \section*{Acknowledgements}
 The authors thank Eva L\"ocherbach for careful reading and valuable comments.

\end{document}